\documentclass[10pt, reqno]{amsart}

\usepackage[centertags]{amsmath}
\usepackage{amsfonts}
\usepackage{amssymb}
\usepackage{amsthm}
\usepackage{newlfont}
\usepackage{fancyhdr}

\def\BibTeX{{\rm B\kern-.05em{\sc i\kern-.025em b}\kern-.08em
    T\kern-.1667em\lower.7ex\hbox{E}\kern-.125emX}}

\newcommand{\EE}{\mathbb{E}}

\newcommand{\E}{\mathrm{E}}
\newcommand{\PP}{\mathrm{P}}
    \newcommand{\dto}{\xrightarrow{d}}
    \newcommand{\wto}{\xrightarrow{w}}
    \newcommand{\vto}{\xrightarrow{v}}
    \newcommand{\fidi}{\xrightarrow{\text{fidi}}}

\newcommand{\eqd}{\stackrel{d}{=}}
 \newcommand{\floor}[1]{\lfloor#1\rfloor}
\theoremstyle{plain}
\newtheorem{thm}{Theorem}[section]

\newtheorem{lem}[thm]{Lemma}

\theoremstyle{definition}
\newtheorem{rem}[thm]{Remark}
\newtheorem{ex}[thm]{Example}
\newtheorem{cond}[thm]{Condition}
\numberwithin{equation}{section}

\begin{document}

\title[Weak convergence of multivariate partial maxima processes] 
{Weak convergence of multivariate partial maxima processes}

%
\author{Danijel Krizmani\'{c}}

\address{Danijel Krizmani\'{c}\\ Department of Mathematics\\
        University of Rijeka\\
        Radmile Matej\v{c}i\'{c} 2, 51000 Rijeka\\
        Croatia}
\email{dkrizmanic@math.uniri.hr}




\begin{abstract}
For a strictly stationary sequence of $\mathbb{R}_{+}^{d}$--valued random vectors we derive functional convergence of partial maxima stochastic processes under joint regular variation and weak dependence conditions. The limit process is an extremal process and the convergence takes place in the space of $\mathbb{R}_{+}^{d}$--valued c\`{a}dl\`{a}g functions on
$[0,1]$, with the Skorohod weak $M_{1}$ topology. We also show that this topology in general can not be replaced by the stronger (standard) $M_{1}$ topology. The theory is illustrated on three examples, including the multivariate squared GARCH process with constant conditional correlations.
\end{abstract}

\maketitle

\section{Introduction}

A classical question in extreme value theory is under what assumptions the scaled maximum
$$\bigvee_{i=1}^{n} \frac{X_{i}-b_{n}}{a_{n}}$$
of i.i.d.~random variables $(X_{i})_{i \in \mathbb{N}}$ converges weakly, for some $a_{n}>0$ and $b_{n} \in \mathbb{R}$. Also what are the possible limit distributions? Answers to these questions were given by Fisher and Tippet~\cite{FT28}, Gnedenko~\cite{Gn43} and de Haan~\cite{Ha70}. Introducing a time variable, Lamperti~\cite{La64} studied the asymptotical distributional behavior of partial maxima stochastic processes
$$ \bigvee_{i=1}^{\lfloor nt \rfloor} \frac{X_{i}-b_{n}}{a_{n}}, \quad t \geq 0.$$
Extension of the theory to dependent random variables, and then to multivariate and spatial settings were particularly stimulating and useful in applications, we refer here only to Adler~\cite{Ad78}, Leadbetter~\cite{Le74},~\cite{Le76}, Beirlant et al.~\cite{BGST04}, de Haan and Ferreira~\cite{HaFe06} and Resnick~\cite{Re07}.

In this paper we focus on the multivariate case in the weakly dependent setting. Let $\mathbb{R}_{+}^{d} = [0, \infty)^{d}$.
We consider a stationary sequence of $\mathbb{R}^{d}_{+}$--valued random vectors $(X_{n})$. In the i.i.d.~case it is well known that weak convergence of the scaled maximum is equivalent to the regular variation of the distribution of $X_{1}$, i.e.
$$ M_{n} = \bigvee_{i=1}^{n}\frac{X_{i}}{a_{n}} \dto Y_{0}$$
if and only if
\begin{equation}\label{e:regvar}
 n \PP \Big( \frac{X_{1}}{a_{n}} \in \cdot\,\Big) \vto \mu(\,\cdot\,),
\end{equation}
where $Y_{0}$ is a random vector with distribution function $F_{0}(x) = e^{-\mu([[0,x]]^{c})}, x \in \mathbb{R}_{+}^{d}$, $\mu$ is a Radon measure and $(a_{n})$ a sequence of positive real numbers such that
$$ n \PP (\|X_{1}\| > a_{n}) \to 1 \qquad \textrm{as} \ n \to \infty,$$
see Proposition 7.1 in Resnick~\cite{Re07}. The arrow $" \vto"$ above denotes vague convergence of measures, and $[[a,b]]$ the product segment, i.e.
$$[[a,b]]=[a^{1},b^{1}] \times [a^{2},b^{2}]
\times \dots \times [a^{d},b^{d}]$$
for $a=(a^{1}, \ldots, a^{d}), b=(b^{1}, \ldots, b^{d}) \in
\mathbb{R}_{+}^{d}$.

In the i.i.d.~case relation (\ref{e:regvar}) is also equivalent to the functional convergence of stochastic processes of partial maxima of $(X_{n})$, i.e.
 \begin{equation}\label{e:functconv}
   M_{n}(\,\cdot\,) = \bigvee_{i=1}^{\lfloor n \cdot
   \rfloor}\frac{X_{i}}{a_{n}} \dto Y_{0}(\,\cdot\,)
 \end{equation}
in $D([0,1], \mathbb{R}_{+}^{d})$, the space of $\mathbb{R}_{+}^{d}$--valued c\`{a}dl\`{a}g functions on
$[0,1]$, with the Skorohod $J_{1}$ topology, with the limit $Y_{0}(\,\cdot\,)$ being an extremal process, see Proposition 7.2 in Resnick~\cite{Re07}.

In this paper we are interested in the investigation of the asymptotic distributional behavior of the processes $M_{n}(\,\cdot\,)$ for a sequence of weakly dependent $\mathbb{R}_{+}^{d}$--valued random vectors that are jointly regularly varying. Since we study extremes of random processes, nonnegativity of the components of random vectors
$X_{n}$ in reality is not a restrictive assumption.

First, we introduce the essential ingredients about regular variation, weak dependence and Skorohod topologies in Section~\ref{s:one}. In Section~\ref{s:two} we prove the so
called timeless result on weak convergence of scaled extremes
$M_{n}$, based on a point process convergence obtained by Davis and Mikosch~\cite{DaMi98}. Using this result and a multivariate version of the limit theorem derived by Basrak et al.~\cite{BKS} for a certain time-space point processes, in Section~\ref{s:three} we prove a functional limit theorem for
processes of partial maxima $M_{n}(\,\cdot\,)$ in the space $D([0,1], \mathbb{R}_{+}^{d})$
endowed with the Skorohod weak $M_{1}$ topology. This topology is weaker than the standard $M_{1}$ topology (when $d >1$). The used methods are partly based on the work of Basrak and Krizmani\'{c}~\cite{BaKr} for partial sums.
Finally, in Section~\ref{s:four} the theory is applied to $m$--dependent processes, stochastic recurrence equations and multivariate squared GARCH (p,q) with constant conditional correlations. We also illustrate by an example that
the weak $M_{1}$ convergence in our main theorem, in general, can not be replaced by the standard $M_{1}$ convergence.

\section{Preliminaries}\label{s:one}

In this section we introduce some basic notions and results on regular variation and point processes that will be used in the following sections.

\subsection{Regular variation}

Regular variation on $\mathbb{R}_{+}^{d} $ for random vectors is typically formulated in terms of vague convergence on $\EE^{d}= [0, \infty]^{d} \setminus \{ \textbf{0} \}$. The topology on $\EE^{d}$ is chosen so that a set $B \subseteq \EE^{d}$
has compact closure if and only if it is bounded away from zero,
that is, if there exists $u > 0$ such that $B \subseteq \EE^{d}_u = \{ x
\in \EE^{d} : \| x \| >u \}$. Here $\| \cdot \|$ denotes the max-norm on $\mathbb{R}_{+}^{d}$, i.e.\
$\displaystyle \| x \|=\max \{ x^{i} : i=1, \ldots , d\}$ where
$x=(x^{1}, \ldots, x^{d}) \in \mathbb{R}_{+}^{d}$. Denote by $C_{K}^{+}(\mathbb{E}^{d})$ the class of all $\mathbb{R}_{+}$--valued continuous functions on $\mathbb{E}^{d}$ with compact support.

The vector $\xi$ with values in $\mathbb{R}_{+}^{d}$ is (multivariate) regularly varying with index $\alpha >0$ if there exists a random vector $\Theta$
on the unit sphere $\mathbb{S}_{+}^{d-1} = \{ x \in \mathbb{R}_{+}^{d} :
\| x \|=1 \}$ in $\mathbb{R}_{+}^{d}$, such that for every $u \in (0,\infty)$
 \begin{equation}\label{e:regvar1}
   \frac{\PP(\|\xi\| > ux,\,\xi / \| \xi \| \in \cdot \, )}{\PP(\| \xi \| >x)}
    \wto u^{-\alpha} \PP( \Theta \in \cdot \,)
 \end{equation}
as $x \to \infty$, where the arrow "$\wto$" denotes weak convergence of finite measures.
 Regular variation can be expressed in terms of vague convergence of measures on $\mathcal{B}(\EE^{d})$:
$$ n \PP ( a_{n}^{-1} \xi \in \cdot\,) \vto \mu (\,\cdot\,),$$
where $(a_{n})$ is a sequence of positive real numbers tending to infinity and $\mu$ is a non-null Radon measure on $\mathcal{B}(\EE^{d})$.

We say that a strictly stationary $\mathbb{R}^{d}_{+}$--valued process $(\xi_{n})_{n \in \mathbb{Z}}$ is \emph{jointly regularly varying} with index
$\alpha >0$ if for any nonnegative integer $k$ the
$kd$-dimensional random vector $\xi = (\xi_{1}, \ldots ,
\xi_{k})$ is multivariate regularly varying with index $\alpha$.

Theorem~2.1 in Basrak and Segers~\cite{BaSe} provides a convenient
characterization of joint regular variation: it is necessary and
sufficient that there exists a process $(Y_n)_{n \in \mathbb{Z}}$
with $\PP(\|Y_0\| > y) = y^{-\alpha}$ for $y \geq 1$ such that as $x
\to \infty$,
\begin{equation}\label{e:tailprocess}
  \bigl( (x^{-1}\ \xi_n)_{n \in \mathbb{Z}} \, \big| \, \| \xi_0\| > x \bigr)
  \fidi (Y_n)_{n \in \mathbb{Z}},
\end{equation}
where "$\fidi$" denotes convergence of finite-dimensional
distributions. The process $(Y_{n})_{n \in \mathbb{Z}}$ is called
the \emph{tail process} of $(\xi_{n})_{n \in \mathbb{Z}}$.

\subsection{Point processes and dependence conditions}

Let $(X_{n})$ be a strictly stationary sequence of $\mathbb{R}^{d}_{+}$--valued random vectors and assume it is jointly regularly varying with index $\alpha >0$. Let $(Y_{n})$ be
the tail process of $(X_{n})$. In order to obtain weak convergence of the scaled extremes $M_{n}$ and the partial maxima processes $M_{n}(\,\cdot\,)$ we will use limit results for the corresponding point processes of jumps and then by the continuous mapping theorem transfer this convergence results to extremes and maxima processes. In order to establish these point process convergence we introduce the following processes
\begin{equation*}\label{E:ppspacetime}
 N_{n} = \sum_{i=1}^{n}\delta_{X_{i}/a_{n}}, \qquad N_{n}^{*} = \sum_{i=1}^{n} \delta_{(i / n,\,X_{i} / a_{n})} \qquad \textrm{for all} \ n\in \mathbb{N},
\end{equation*}
where $(a_{n})$ is a sequence of positive real numbers such that
\begin{equation}\label{e:niz}
 n \PP( \| X_{1}\| > a_{n}) \to 1,
\end{equation}
as $n \to \infty$. The point process convergence for the sequence $(N_{n})$ was obtained by Davis and Mikosch~\cite{DaMi98}, while the convergence for the sequence $(N_{n}^{*})$ in the univariate case was established by Basrak et al.~\cite{BKS}, but with straightforward adjustments it carries over to the multivariate case, see Theorem~\ref{t:convpp} below. The appropriate weak dependence conditions for this convergence results are given below. With them we will be able to control the dependence in the sequence $(X_{n})$.

\begin{cond}\label{c:mixcond1}
There exists a sequence of positive integers $(r_{n})$ such that $r_{n} \to \infty $ and $r_{n} / n \to 0$ as $n \to \infty$ and such that for every $f \in C_{K}^{+}([0,1] \times \mathbb{E}^{d})$, denoting $k_{n} = \lfloor n / r_{n} \rfloor$, as $n \to \infty$,
\begin{equation}\label{e:mixcon}
 \E \biggl[ \exp \biggl\{ - \sum_{i=1}^{n} f \biggl(\frac{i}{n}, \frac{X_{i}}{a_{n}}
 \biggr) \biggr\} \biggr]
 - \prod_{k=1}^{k_{n}} \E \biggl[ \exp \biggl\{ - \sum_{i=1}^{r_{n}} f \biggl(\frac{kr_{n}}{n}, \frac{X_{i}}{a_{n}} \biggr) \biggr\} \biggr] \to 0.
\end{equation}
\end{cond}
It can be shown that Condition~\ref{c:mixcond1} is implied by the strong mixing property (cf. Krizmani\'{c}~\cite{Kr16}). Condition~\ref{c:mixcond1} is slightly stronger than the condition $\mathcal{A}(a_{n})$ introduced by Davis and Mikosch~\cite{DaMi98}.

\begin{cond}\label{c:mixcond2}
There exists a sequence of positive integers $(r_{n})$ such that $r_{n} \to \infty $ and $r_{n} / n \to 0$ as $n \to \infty$ and such that for every $u > 0$,
\begin{equation}
\label{e:anticluster}
  \lim_{m \to \infty} \limsup_{n \to \infty}
  \PP \biggl( \max_{m \leq |i| \leq r_{n}} \| X_{i} \| > ua_{n}\,\bigg|\,\| X_{0}\|>ua_{n} \biggr) = 0.
\end{equation}
\end{cond}

By Proposition~4.2 in Basrak and Segers~\cite{BaSe},
under Condition~\ref{c:mixcond2} the following
holds
\begin{eqnarray}\label{E:theta:spectral}
   \theta = \PP ({\textstyle\sup_{i\ge 1}} \| Y_{i}\| \le 1) = \PP ({\textstyle\sup_{i\le -1}} \| Y_{i}\| \le 1)>0,
\end{eqnarray}
and $\theta$ is the extremal index of the univariate sequence $(\| X_{n} \|)$.
Recall that a strictly stationary sequence of nonnegative random variables $(\xi_{n})$ has extremal index $\theta$ if for every $\tau >0$ there exists a sequence of real numbers $(u_{n})$ such that
\begin{equation}\label{e:eindex}
 \lim_{n \to \infty} n \PP( \xi_{1} > u_{n}) \to \tau \qquad \textrm{and} \qquad \lim_{n \to \infty} \PP \bigg( \max_{1 \leq i \leq n} \xi_{i} \leq u_{n} \bigg) \to e^{-\theta \tau}.
\end{equation}
It holds that $\theta \in [0,1]$. In particular, if the $\xi_{n}$ are i.i.d. then (\ref{e:eindex}) can hold only for $\theta =1$.
For a detailed discussion on joint regular variation and dependence Conditions~\ref{c:mixcond1} and \ref{c:mixcond2} we refer to Basrak et al.~\cite{BKS}, Section 3.4.

Under joint regular variation and Conditions~\ref{c:mixcond1} and \ref{c:mixcond2}, by Theorem 2.8 in Davis and Mikosch~\cite{DaMi98} we obtain the convergence in distribution of point processes $N_{n}$ to some $N$, which by Theorem 2.2 and Corollary 2.4 in~\cite{DaMi98} has the following cluster representation
\begin{equation}\label{eq:Nnconv}
 N \eqd \sum_{i} \sum_{j} \delta_{P_{i}Q_{ij}},
 \end{equation}
where $\sum_{i=1}^{\infty}\delta_{P_{i}}$ is a Poisson process on $\mathbb{R}_{+}$
with intensity measure $\kappa$ given by $\kappa(dy) = \theta \alpha
y^{-\alpha-1}1_{(0,\infty)}(y)\,dy$, and $\sum_{j= 1}^{\infty}\delta_{Q_{ij}}$, $i \geq 1$,
are i.i.d. point processes whose points satisfy $\sup_{j}\|Q_{ij}\|=1$, and all point processes are mutually
independent. For a more precise description of the distribution of point process $\sum_{j= 1}^{\infty}\delta_{Q_{ij}}$ see~\cite{DaMi98}.

Then by the same arguments as in the proof of Theorem 2.3 in~\cite{BKS} one obtains the following result (cf. also Basrak and Krizmani\'{c}~\cite{BaKr}).

\begin{thm}\label{t:convpp}
\label{T:pointprocess:complete} Assume that Conditions~\ref{c:mixcond1}
and \ref{c:mixcond2} hold
for the same sequence $(r_n)$. Then for every $u \in (0, \infty)$ and as
$n \to \infty$,
\begin{equation}\label{eq:convpp}
    N_n^{*} \bigg|_{[0, 1] \times \EE^{d}_u}\, \dto N^{(u)}
    = \sum_i \sum_j \delta_{(T^{(u)}_i, u Z_{ij})} \bigg|_{[0, 1] \times \EE^{d}_u}\,,
\end{equation}
in $[0, 1] \times \EE^{d}_u$ and
\begin{enumerate}
\item $\sum_i \delta_{T^{(u)}_i}$ is a homogeneous Poisson process on $[0, 1]$ with intensity $\theta u^{-\alpha}$,
\item $(\sum_j \delta_{Z_{ij}})_i$ is an i.i.d.~sequence of point processes in $\EE^{d}$, independent of $\sum_i \delta_{T^{(u)}_i}$, and with distribution equal to  $( \sum_{j \in \mathbb{Z}} \delta_{Y_j} \,|\, \sup_{i \le -1} \| Y_i\| \le 1).$
\end{enumerate}
\end{thm}

\subsection{The weak $M_{1}$ topology}\label{ss:j1m1}

The stochastic processes that we consider have discontinuities, and therefore it is natural for the function space of sample paths of these stochastic processes to take the space $D([0,1], \mathbb{R}_{+}^{d})$ of all right-continuous $\mathbb{R}_{+}^{d}$--valued functions on $[0,1]$ with left limits.

In the one dimensional case (cf. Krizmani\'{c}~\cite{Kr14}) the partial maxima processes $M_{n}(\,\cdot\,)$ converge to an extremal process in the space $D([0,1], \mathbb{R}_{+})$ equipped with the Skorohod $M_{1}$ topology. In this paper we extend this result to the multivariate setting, but with the weak $M_{1}$ topology, since as we show later the direct generalization of the one-dimensional result to random vectors fails in the standard $M_{1}$ topology on $D([0,1], \mathbb{R}_{+}^{d})$ for $d \geq 2$. In the sequel we give the definition of the weak $M_{1}$ topology.

 For $x \in D([0,1],
\mathbb{R}_{+}^{d})$ the completed graph of $x$ is the set
\[
  G_{x}
  = \{ (t,z) \in [0,1] \times \mathbb{R}_{+}^{d} : z \in [[x(t-), x(t)]]\},
\]
where $x(t-)$ is the left limit of $x$ at $t$. We define an
order on the graph $G_{x}$ by saying that $(t_{1},z_{1}) \le
(t_{2},z_{2})$ if either (i) $t_{1} < t_{2}$ or (ii) $t_{1} = t_{2}$
and $|x^{j}(t_{1}-) - z^{j}_{1}| \le |x^{j}(t_{2}-) - z^{j}_{2}|$
for all $j=1,\ldots,d$. Note that the relation $\le$ induces only a partial
order on the graph $G_{x}$. A weak parametric representation
of the graph $G_{x}$ is a continuous nondecreasing function $(r,u)$
mapping $[0,1]$ into $G_{x}$, with $r \in C([0,1],[0,1])$ being the
time component and $u=(u^{1},\ldots, u^{d}) \in C([0,1],
\mathbb{R}_{+}^{d})$ being the spatial component, such that $r(0)=0,
r(1)=1$ and $u(1)=x(1)$. Let $\Pi_{w}(x)$ denote the set of weak
parametric representations of the graph $G_{x}$. For $x_{1},x_{2}
\in D([0,1], \mathbb{R}_{+}^{d})$ define
\[
  d_{w}(x_{1},x_{2})
  = \inf \{ \|r_{1}-r_{2}\|_{[0,1]} \vee \|u_{1}-u_{2}\|_{[0,1]} : (r_{i},u_{i}) \in \Pi_{w}(x_{i}), i=1,2 \},
\]
where $\|x\|_{[0,1]} = \sup \{ \|x(t)\| : t \in [0,1] \}$. Now we
say that $x_{n} \to x$ in $D([0,1], \mathbb{R}_{+}^{d})$ for a sequence
$(x_{n})$ in the weak Skorohod $M_{1}$ (or shortly $WM_{1}$)
topology if $d_{w}(x_{n},x)\to 0$ as $n \to \infty$. The $WM_{1}$
topology is weaker than the standard (or strong) $M_{1}$ topology on $D([0,1],
\mathbb{R}_{+}^{d})$. For $d=1$ the two topologies coincide. The $WM_{1}$ topology
coincides with the topology induced by the metric
\begin{equation}\label{e:defdp}
 d_{p}(x_{1},x_{2})=\max \{ d_{M_{1}}(x_{1}^{j},x_{2}^{j}) :
j=1,\ldots,d\}
\end{equation}
 for $x_{i}=(x_{i}^{1}, \ldots, x_{i}^{d}) \in D([0,1],
 \mathbb{R}_{+}^{d})$ and $i=1,2$ (here $d_{M_{1}}$ denotes the standard Skorohod
 $M_{1}$ metric on $D([0,1],\mathbb{R}_{+})$). The metric $d_{p}$ induces the product topology on $D([0,1], \mathbb{R}_{+}^{d})$.
For detailed discussion of the strong and weak $M_{1}$ topologies we refer to
Whitt~\cite{Whitt02}, sections 12.3--12.5. Recall here the definition of the metric $d_{M_{1}}$. For $x \in D([0,1], \mathbb{R}_{+}^{d})$
we define the set
\[
  \Gamma_{x}
  = \{ (t,z) \in [0,1] \times \mathbb{R}_{+}^{d} : z \in [x(t-), x(t)] \},
\]
where $[a,b] = \{  \lambda a + (1-\lambda)b : 0 \leq \lambda \leq 1 \}$ for $a, b \in \mathbb{R}_{+}^{d}$. We say $(r,u)$ is a parametric representation of $\Gamma_{x}$ if it is a continuous nondecreasing function mapping $[0,1]$ onto $\Gamma_{x}$. Denote by $\Pi(x)$ the set of all parametric representations of the graph $\Gamma_{x}$. Then for $x_{1},x_{2} \in D([0,1], \mathbb{R}_{+}^{d})$
\[
  d_{M_{1}}(x_{1},x_{2})
  = \inf \{ \|r_{1}-r_{2}\|_{[0,1]} \vee \|u_{1}-u_{2}\|_{[0,1]} : (r_{i},u_{i}) \in \Pi(x_{i}), i=1,2 \}.
\]

\section{Weak convergence of partial maxima $M_{n}$}\label{s:two}

In this section we establish weak convergence of the multivariate partial maxima $M_{n}$ by generalizing the corresponding one dimensional result given in Krizmani\'{c}~\cite{Kr14}. Let $(X_{n})$ be a strictly stationary sequence of $\mathbb{R}^{d}_{+}$--valued random vectors, jointly regularly varying with index $\alpha \in (0, \infty)$ and assume Conditions~\ref{c:mixcond1} and \ref{c:mixcond2} hold. Then by (\ref{eq:Nnconv}) it holds that, as $n \to \infty$,
$$ N_{n} = \sum_{i=1}^{n}\delta_{X_{i}/a_{n}} \dto N = \sum_{i} \sum_{j} \delta_{P_{i}Q_{ij}},$$
where $(a_{n})$ is chosen as in (\ref{e:niz}).
Denote by
$\mathbf{M}_{p}(\EE^{d})$ the space of Radon point measures on $\EE^{d}$
equipped with the vague topology. Recall $M_{n} = a_{n}^{-1}
\bigvee_{i=1}^{n} X_{i} = \big(a_{n}^{-1} \bigvee_{i=1}^{n} X_{i}^{k}\big)_{k=1,\ldots,d}$.

\begin{thm}\label{t:notimeconv}
Let $(X_{n})$ be a strictly stationary sequence of $\mathbb{R}_{+}^{d}$--valued random vectors, jointly regularly varying with index $\alpha\in(0, \infty)$. Suppose that Conditions~\ref{c:mixcond1} and \ref{c:mixcond2} hold. Then, as $n \to \infty$,
$$ M_{n} \dto M = \bigvee_{i}\bigvee_{j}P_{i}Q_{ij}.$$
\end{thm}

\begin{proof}
Let $\epsilon >0$ be arbitrary. The mapping
$T_{\epsilon} \colon \mathbf{M}_{p}(\EE^{d}) \to \mathbb{R}_{+}^{d}$ defined by
$$ T_{\epsilon} \Big( \sum_{i=1}^{\infty}\delta_{x_{i}} \Big) =
\bigg( \bigvee_{i=1}^{\infty}x_{i}^{k} 1_{\{ x_{i}^{k} \in [\epsilon, \infty)
\}} \bigg)_{k=1,\ldots,d}$$
is continuous on the set
$$\Lambda_{\epsilon} = \{ \eta \in
\mathbf{M}_{p}(\EE^{d}) : \eta (\{ (y_{1}, \ldots, y_{d}) : y_{i}= \epsilon \ \textrm{for some} \ i \})=0 \}.$$
One can see this by showing the continuity of the components
$$ T_{\epsilon}^{k} \Big( \sum_{i=1}^{\infty}\delta_{x_{i}^{k}} \Big) = \bigvee_{i=1}^{\infty}x_{i}^{k} 1_{\{ x_{i}^{k} \in [\epsilon, \infty)
\}}$$ (cf. the one dimensional case in Krizmani\'{c}~\cite{Kr14}).

$N$ has no fixed atoms (see Lemma 2.1 in
Davis and Mikosch~\cite{DaMi98}), i.e. $\PP (N \in \Lambda_{\epsilon})=1$, and therefore by the
continuous mapping theorem we get
\begin{equation}\label{e:convTeps}
M_{n}[\epsilon, \infty)= T_{\epsilon}(N_{n}) \dto T_{\epsilon}(N) = M[\epsilon, \infty) \qquad \textrm{as} \ n \to \infty,
\end{equation}
with the notation
$$ M_{n}B = (M_{n}^{k} B)_{k=1,\ldots,d} =  \bigg( a_{n}^{-1} \bigvee_{i=1}^{n}X_{i}^{k} 1_{\{ a_{n}^{-1}X_{i}^{k} \in B \}} \bigg)_{k=1,\ldots,d},$$
and
$$ M B = (M^{k} B)_{k=1,\ldots,d} = \bigg( \bigvee_{i=1}^{\infty} \bigvee_{j=1}^{\infty}P_{i}Q_{ij}^{k} 1_{\{ P_{i}Q_{ij}^{k} \in B \}} \bigg)_{k=1,\ldots,d}$$
for any Borel set $B$ in $\mathbb{R}_{+}$.
Obviously
\begin{equation}\label{e:convas}
M[\epsilon, \infty) \to M(0,\infty) = M
\end{equation}
almost surely as $\epsilon \to 0$.

In order to obtain $M_{n} \dto M$, i.e. $M_{n}(0,\infty) \dto M(0,\infty)$ as $n \to \infty$, by Theorem 3.5 in Resnick~\cite{Re07} it suffices to prove
that
\begin{equation}\label{eq:slutskycond}
 \lim_{\epsilon \to 0} \limsup_{n \to \infty} \PP (\| M_{n}[\epsilon,\infty) - M_{n}(0,\infty) \| > \delta)=0
 \end{equation}
for any $\delta >0$.
Since for arbitrary real numbers $x_{1}, \ldots, x_{n}, y_{1},
\ldots, y_{n}$ the following inequality
 \begin{equation}\label{e:maxineq}
 \Big| \bigvee_{i=1}^{n}x_{i} - \bigvee_{i=1}^{n}y_{i} \Big| \leq
\bigvee_{i=1}^{n}|x_{i}-y_{i}|
 \end{equation}
  holds, note that
\begin{equation*}
  |M_{n}^{k}[\epsilon,\infty) - M_{n}^{k}(0,\infty)|  \leq M_{n}^{k}(0,\epsilon)
\end{equation*}
for all $k=1,\ldots,d$, and this yields
\begin{equation}\label{eq:slutskypom}
\| M_{n}[\epsilon,\infty) - M_{n}(0,\infty) \| = \bigvee_{k=1}^{d} |M_{n}^{k}[\epsilon,\infty) - M_{n}^{k}(0,\infty)|  \leq \| M_{n}(0,\epsilon)\|.
\end{equation}

Take now an arbitrary $s > \alpha$. Then using stationarity and Markov's inequality we get the bound
\begin{eqnarray}\label{eq:slutsky1}
   \nonumber \PP ( \|M_{n}(0,\epsilon)\| > \delta) & \leq & n \PP \bigg( \bigvee_{k=1,\ldots,d} \frac{X_{1}^{k}}{a_{n}} 1_{\{X_{1}^{k} < \epsilon a_{n}\}} > \delta \bigg)\\[0.4em]
    \nonumber & \hspace*{-10em} \leq & \hspace*{-5em} n \sum_{k=1}^{d} \PP \bigg( \frac{X_{1}^{k}}{a_{n}} 1_{\{X_{1}^{k} < \epsilon a_{n}\}} > \delta \bigg) \leq n \sum_{k=1}^{d} \frac{1}{\delta^{s} a_{n}^{s}} \E ( (X_{1}^{k})^{s} 1_{\{ X_{1}^{k} < \epsilon a_{n} \}})\\[0.4em]
    \nonumber & \hspace*{-10em} = & \hspace*{-5em} \frac{n}{\delta^{s} a_{n}^{s}} \sum_{k=1}^{d} \Big[ \E ( (X_{1}^{k})^{s} 1_{ \{X_{1}^{k} < \epsilon a_{n}, \|X_{1}\| > \epsilon a_{n} \} }) + \E ( (X_{1}^{k})^{s} 1_{ \{ X_{1}^{k} < \epsilon a_{n}, \|X_{1}\| \leq \epsilon a_{n} \} }) \Big]\\[0.4em]
    \nonumber & \hspace*{-10em} \leq & \hspace*{-5em} \frac{n}{\delta^{s}} \sum_{k=1}^{d} \Big[ \epsilon^{s} \PP (\|X_{1}\| > \epsilon a_{n} ) + \E \Big( \frac{\|X_{1}\|^{s}}{a_{n}^{s}} 1_{\{ \|X_{1}\| \leq \epsilon a_{n} \} } \Big) \Big]\\[0.4em]
     & \hspace*{-10em} = &  \hspace*{-5em} \frac{\epsilon^{s}d}{\delta^{s}} \cdot n \PP( \|X_{1}\| >  \epsilon a_{n}) \bigg[ 1 + \frac{\E (\|X_{1}\|^{s} 1_{\{ \|X_{1}\| < \epsilon a_{n} \}})}{\epsilon^{s} a_{n}^{s} \PP(\|X_{1}\|>\epsilon a_{n})} \bigg].
 \end{eqnarray}
Since the distribution of $\|X_{1}\|$ is regularly varying with index
$\alpha$, using (\ref{e:niz}) it follows immediately that
$$ n \PP(\|X_{1}\|> \epsilon a_{n}) \to \epsilon^{-\alpha}$$
 as $n \to \infty$. By Karamata's theorem
 $$ \lim_{n \to \infty} \frac{\E(\|X_{1}\|^{s} \, 1_{ \{ \|X_{1}\| < \epsilon a_{n} \}
            })}{\epsilon^{s} a_{n}^{s}\PP(\|X_{1}\|> \epsilon a_{n})} =
            \frac{\alpha}{s-\alpha}.$$
  Thus from (\ref{eq:slutsky1}) we get
   $$ \limsup_{n \to \infty} \PP(\|M_{n}(0,\epsilon)\|>\delta) \leq \frac{\epsilon^{s-\alpha}d}{\delta^{s}} \Big[ 1 +
 \frac{\alpha}{s-\alpha} \Big].$$
 Letting $\epsilon \to 0$ we finally obtain
 $$ \lim_{\epsilon \to 0} \limsup_{n \to \infty} \PP(\|M_{n}(0,\epsilon)\|>\delta) = 0,$$
 and taking into account (\ref{eq:slutskypom}), relation (\ref{eq:slutskycond}) follows. Hence $ M_{n} \dto M$ as $n \to \infty$.
\end{proof}

\begin{rem}
From the representation in (\ref{eq:Nnconv}) and the fact that $\sup_{j}\|Q_{ij}\|=1$ it follows that $\|M\|$ is a Fr\'{e}chet random variable, since
\begin{eqnarray*}
  \PP(\|M\| \leq x) &=& \PP \bigg( \max_{k=1,\ldots,d} \bigvee_{i}\bigvee_{j}P_{i}Q_{ij}^{k} \leq x \bigg) = \PP \bigg( \bigvee_{i}P_{i} \leq x \bigg) \\[0.4em]
   &=& \PP \Big( \sum_{i}\delta_{P_{i}}(x,\infty) = 0 \Big) = e^{-\kappa(x,\infty)} = e^{-\theta x^{-\alpha}}
\end{eqnarray*}
 for $x>0$.
\end{rem}

\section{Functional convergence of partial maxima processes $M_{n}(\,\cdot\,)$}\label{s:three}

In this section we show the convergence of the partial maxima
process
\begin{equation*}
  M_{n}(t) =
  \bigvee_{i=1}^{\floor{nt}} \frac{X_{i}}{a_{n}} = \bigg( \bigvee_{i=1}^{\floor{nt}} \frac{X_{i}^{k}}{a_{n}}\bigg)_{k=1,\ldots,d}, \quad t \in [0,1],
\end{equation*}
to an extremal process in the space $D([0,1], \mathbb{R}_{+}^{d})$
equipped with Skorohod weak $M_1$ topology. Similar to the one dimensional case treated in Krizmani\'{c}~\cite{Kr14} we first represent $M_n(\,\cdot\,)$ as the image of the
time-space point process $N_n^{*}$ under a certain maximum
functional. Then, using certain continuity properties of this
functional, the continuous mapping theorem and the standard
"finite dimensional convergence plus tightness" procedure we
transfer the weak convergence of $N_n^{*}$ in (\ref{eq:convpp}) to
weak convergence of $M_n(\,\cdot\,)$.

Extremal processes can be derived from Poisson processes in the following way. Let $\xi = \sum_{k}\delta_{(t_{k}, j_{k})}$ be a Poisson process on $[0,\infty) \times \EE^{d}$ with mean measure $\lambda \times \nu$, where $\lambda$ is the Lebesgue measure and $\nu$ is a measure on $\EE^{d}$ satisfying
$$\nu (\{ x \in \EE^{d} : \|x\| > \delta \}) < \infty$$
for any $\delta >0$. The extremal process $\widetilde{M}(\,\cdot\,)$ generated by $\xi$ is defined by
$$ \widetilde{M}(t) = \bigvee_{t_{k} \leq t}j_{k}, \qquad t>0.$$
Then for $x \in \mathbb{E}^{d}$ and $t>0$ it holds that
$$ \PP ( \widetilde{M}(t) \leq x) = e^{-t \nu([[0,x]]^{c})},$$
with the notation that for two vectors $y=(y^{1}, \ldots, y^{d})$ and $z=(z^{1}, \ldots, z^{d})$, $y \leq z$ means $y^{k} \leq z^{k}$ for all $k=1,\ldots, d$
(cf. Resnick~\cite{Re07}, section 5.6). The measure $\nu$ is called the exponent measure.

Now fix $0 < v < u < \infty$ and define the maximum functional
$$
  \phi^{(u)} \colon \mathbf{M}_{p}([0,1] \times \EE^{d}_{v}) \to
  D([0,1], \mathbb{R}_{+}^{d})
$$
 by
 $$ \phi^{(u)} \Big( \sum_{i}\delta_{(t_{i}, (x_{i}^{1}, \ldots, x_{i}^{d}))} \Big) (t)
  =  \Big( \bigvee_{t_{i} \leq t} x_{i}^{k} \,1_{\{u < x_i^{k} < \infty\}} \Big)_{k=1,\ldots,d}, \qquad t \in [0,
  1],$$
  where the supremum of an empty set may be taken, for
  convenience, to be $0$.
 $\phi^{(u)}$ is well defined because $[0,1] \times
\EE^{d}_{u}$ is a relatively compact subset of $[0,1] \times \EE^{d}_{v}$.
The space $\mathbf{M}_p([0,1] \times \EE^{d}_{v})$ of Radon point
measures on $[0,1] \times \EE^{d}_{v}$ is equipped with the vague
topology and $D([0,1], \mathbb{R}_{+}^{d})$ is equipped with the weak $M_1$ topology. Let
\begin{multline*}
  \Lambda =
  \{ \eta \in \mathbf{M}_{p}([0,1] \times \EE^{d}_{v}) :
    \eta ( \{0,1 \} \times \EE^{d}_{u}) = 0 \ \textrm{and} \\[0.3em]
     \eta ([0,1] \times \{ x=(x^{1},\ldots,x^{d}) : x^{i} \in \{u, \infty \} \ \textrm{for some} \ i \}) =0 \}.
\end{multline*}
 Then the point process $N^{(v)}$ defined in
(\ref{eq:convpp}) almost surely belongs to the set $\Lambda$, see
Lemma 3.1 in Basrak and Krizmani\'{c}~\cite{BaKr}. Now we will show that
$\phi^{(u)}$ is continuous on the set $\Lambda$.

\begin{lem}\label{l:contfunct}
The maximum functional $\phi^{(u)} \colon \mathbf{M}_{p}([0,1]
\times \EE^{d}_{v}) \to D([0,1], \mathbb{R}_{+}^{d})$ is continuous on the set $\Lambda$,
when $D([0,1], \mathbb{R}_{+}^{d})$ is endowed with the weak $M_{1}$ topology.
\end{lem}

\begin{proof}
Take an arbitrary $\eta \in \Lambda$ and suppose that $\eta_{n} \vto \eta$ in $\mathbf{M}_p([0,1] \times
\EE_{v}^{d})$. We need to show that
$\phi^{(u)}(\eta_n) \to \phi^{(u)}(\eta)$ in $D([0,1],
\mathbb{R}_{+}^{d})$ according to the $WM_1$ topology. By
Theorem~12.5.2 in Whitt~\cite{Whitt02}, it suffices to prove that,
as $n \to \infty$,
$$ d_{p}(\phi^{(u)}(\eta_{n}), \phi^{(u)}(\eta)) =
\max_{k=1,\ldots,d}d_{M_{1}}(\phi^{(u)\,k}(\eta_{n}),
\phi^{(u)\,k}(\eta)) \to 0,$$
 where $\phi^{(u)}(\xi)=(\phi^{(u)\,k}(\xi))_{k=1,\ldots,d}$ for
 $\xi \in \mathbf{M}_{p}([0,1] \times \EE_{v}^{d})$.

 Now one can follow, with small modifications, the lines in the proof of Lemma~4.1 in Krizmani\'{c}~\cite{Kr14} to obtain
$d_{M_{1}}(\phi^{(u)\,k}(\eta_{n}), \phi^{(u)\,k}(\eta)) \to 0$ as
$n \to \infty$. Therefore $d_{p}(\phi^{(u)}(\eta_{n}),
\phi^{(u)}(\eta)) \to 0$ as $n \to \infty$, and we conclude that
$\phi^{(u)}$ is continuous at $\eta$.
\end{proof}

\smallskip

\begin{lem}\label{l:fddmultconv}
 Assume $\xi_{n}=\sum_{i}\delta_{(t_{i}^{(n)}, j_{i}^{(n)})}$, $n \geq 0$, are Poisson processes on $[0, \infty) \times \EE^{d}$ with mean measures $\lambda \times \beta_{n}$, and let $H_{n}$ be the corresponding extremal processes generated by the $\xi_{n}$'s. If
 \begin{equation}\label{e:mconvexpr}
 \beta_{n} \vto \beta_{0} \qquad \textrm{as} \ n \to \infty,
 \end{equation}
  then the finite dimensional distributions of $H_{n}(\,\cdot\,)$ converge to the finite dimensional distributions of $H_{0}(\,\cdot\,)$ as $n \to \infty$.
\end{lem}

\begin{proof}
By Lemma 6.1 in Resnick~\cite{Re07}, from (\ref{e:mconvexpr}) we obtain that, as $n \to \infty$,
\begin{equation}\label{e:mconvexpr2}
\beta_{n}([[0,x]]^{c}) \to \beta_{0}([[0,x]]^{c})
\end{equation}
for all continuity points $x$ of $\beta_{0}([[0,\,\cdot\,]]^{c})$.

Similar to the univariate case, the finite dimensional distributions of $H_{n}(\,\cdot\,) = \bigvee_{t_{i}^{(n)} \leq \, \cdot}j_{i}^{(n)}$ are of the form
\begin{eqnarray*}
   \PP(H_{n}(t_{1}) \leq x_{1}, \ldots, H_{n}(t_{m}) \leq x_{m}) & &  \\[0.3em]
   & \hspace*{-28em} = & \hspace*{-14em} e^{-t_{1}\beta_{n}([[0, \bigwedge_{i=1}^{m}x_{i}]]^{c})} \cdot e^{-(t_{2}-t_{1}) \beta_{n}([[0, \bigwedge_{i=2}^{m}x_{i}]]^{c})} \cdot \ldots \cdot e^{-(t_{m}-t_{m-1})\beta_{n}([[0, x_{m}]]^{c})},
\end{eqnarray*}
for $0 \leq t_{1} < t_{2} < \ldots < t_{m} \leq 1$ and $x_{1},\ldots, x_{m} \in \mathbb{E}^{d}$.
Letting $n \to \infty$ and using (\ref{e:mconvexpr2}) we immediately obtain that the right hand side in the last equation above converges (in the continuity points $x_{1}, \ldots, x_{m}$ of $\beta_{0}([[0,\,\cdot]]^{c})$) to
$$ e^{-t_{1} \beta_{0}([[0, \bigwedge_{i=1}^{m}x_{i}]]^{c})} \cdot e^{-(t_{2}-t_{1}) \beta_{0}([[0, \bigwedge_{i=2}^{m}x_{i}]]^{c})} \cdot \ldots \cdot e^{-(t_{m}-t_{m-1}) \beta_{0}([[0, x_{m}]]^{c})}.$$
But since this limit is in fact $\PP(H_{0}(t_{1}) \leq x_{1}, \ldots, H_{0}(t_{m}) \leq x_{m})$, we conclude that the finite dimensional distributions of $H_{n}(\,\cdot\,)$ converge to the finite dimensional distributions of $H_{0}(\,\cdot\,)$ as $n \to \infty $.
\end{proof}

\begin{thm}\label{t:functconvergence}
Let $(X_{n})$ be a strictly stationary sequence of $\mathbb{R}^{d}_{+}$--valued random vectors, jointly regularly varying with index $\alpha >0$. Suppose that Conditions~\ref{c:mixcond1} and~\ref{c:mixcond2} hold. Then the partial maxima stochastic process
$$  M_{n}(t) = \bigvee_{i=1}^{\lfloor nt
   \rfloor}\frac{X_{i}}{a_{n}}, \qquad t \in [0,1],$$
satisfies
$$ M_{n}(\,\cdot\,) \dto \widetilde{M}(\,\cdot\,) \qquad \textrm{as} \ n \to \infty,$$
in $D([0,1], \mathbb{R}_{+}^{d})$ endowed with the weak $M_{1}$ topology, where
$\widetilde{M}(\,\cdot\,)$ is an extremal process.
\end{thm}

\begin{rem}
The exponent measure $\nu$ of the limiting process $\widetilde{M}(\,\cdot\,)$ in the theorem is the vague limit of the sequence of measures $(\nu^{(u)})$ ($u >0$) as $u \downarrow 0$, with $\nu^{(u)}$ being defined by
\begin{equation*}
\label{E:nuu}
 \begin{array}{rl}
 \nu^{(u)}(((x, y]]) & = \displaystyle u^{-\alpha} \, \PP \biggl( u \bigvee_{i \ge 0} \big( Y_i^{j} \, 1_{\{Y_i^{j} > 1\}} \big)_{j=1,\ldots,d} \in ((x,y]], \, \sup_{i \le -1} \|Y_i\| \le 1
  \biggr),
  \end{array}
\end{equation*}
for $x=(x^{1},\ldots,x^{d}),\,y=(y^{1},\ldots,y^{d}) \in \EE^{d}$ such
that $((x,y]]=(x^{1},y^{1}] \times \dots \times (x^{d},y^{d}]$ is
bounded away from zero. Here $(Y_{n})$ is the tail process of the sequence $(X_{n})$.
\end{rem}

\begin{proof} (\emph{Theorem~\ref{t:functconvergence}})
Using the techniques from the proof of Theorem 3.4 in Basrak and Krizmani\'{c}~\cite{BaKr} we obtain that the point process
$$ \widehat{N}^{(u)} =  \sum_{i} \delta_{(T_{i}^{(u)},\,u \bigvee_{j} (Z_{ij}^{k}1_{\{ Z_{ij}^{k}>1
 \}})_{k=1,\ldots,d})}$$
is a Poisson process with mean measure $\lambda \times \nu^{(u)}$.

Consider now $0<v<u$ and
$$  \phi^{(u)} (N_{n}^{*}\,|\,_{[0,1] \times \EE_{u}^{d}}) (\,\cdot\,)
  = \phi^{(u)} (N_{n}^{*}\,|\,_{[0,1] \times \EE_{v}^{d}}) (\,\cdot\,)
  = \bigvee_{i/n \le \, \cdot} \Big( \frac{X_{i}^{k}}{a_{n}} 1_{ \big\{ \frac{X_{i}^{k}}{a_{n}} > u
    \big\} } \Big)_{k=1,\ldots,d},$$
which by Theorem~\ref{t:convpp}, Lemma~\ref{l:contfunct} and the continuous mapping theorem converges in distribution in $D([0,1], \mathbb{R}_{+}^{d})$
under the $WM_{1}$ topology to
$$
\phi^{(u)} (N^{(v)})(\,\cdot\,)
\eqd \phi^{(u)} (N^{(v)}\,|\,_{[0,1] \times \EE_{u}^{d}})(\,\cdot\,) \eqd \bigvee_{T_{i}^{(u)} \le \, \cdot}
   \bigvee_{j} u (Z_{ij}^{k}1_{ \{ Z_{ij}^{k} > 1 \} })_{k=1,\ldots,d}.
$$
This can be rewritten as
\begin{equation}\label{eq:convaboveu}
  M_{n}^{(u)}(\,\cdot\,) := \bigvee_{i = 1}^{\lfloor n \, \cdot \, \rfloor} \Big( \frac{X_{i}^{k}}{a_{n}} 1_{ \big\{ \frac{X_{i}^{k}}{a_{n}} > u
    \big\} } \Big)_{k=1,\ldots,d} \dto M^{(u)}(\,\cdot\,) := \bigvee_{T_{i} \leq \, \cdot} K_{i}^{(u)} \quad \text{as} \ n \to \infty,
 \end{equation}
 in $D([0,1], \mathbb{R}_{+}^{d})$ under the $WM_{1}$ metric, since
 $\phi^{(u)}(N^{(u)}) = \phi^{(u)} (\widehat{N}^{(u)})\,\eqd\,\phi^{(u)} (\widetilde{N}^{(u)})$,
 where
 $$ \widetilde{N}^{(u)} = \sum_{i} \delta_{(T_{i},\,K_{i}^{(u)})}
 $$
 is a Poisson process with mean measure $\lambda \times \nu^{(u)}$.

 Note that the limiting process $M^{(u)}(\,\cdot\,)$ is an extremal process with exponent measure $\nu^{(u)}$, and therefore
 \begin{equation}\label{e:extrprdistr}
\PP (M^{(u)}(t) \leq x) = \PP(\widetilde{N}^{(u)}((0,t] \times [[0,x]]^{c})=0) = e^{-t\nu^{(u)}([[0,x]]^{c})}
\end{equation}
for $t \in [0,1]$ and $x \in \EE^{d}$. Since the function $\pi \colon D([0,1], \mathbb{R}_{+}^{d}) \to \mathbb{R}_{+}^{d}$ defined by $\pi(y)=y(1)$ is continuous (see Theorem 12.5.2 (iii) in Whitt~\cite{Whitt02}), an application of the continuous mapping theorem to relation (\ref{eq:convaboveu}) yields
\begin{equation}\label{e:convMu1}
 M_{n}^{(u)}(1) \dto M^{(u)}(1) \qquad  \textrm{as} \ n \to \infty.
\end{equation}
 Apply now the notation from the proof of Theorem~\ref{t:notimeconv} to see that $M_{n}^{(u)}(1) = M_{n}(u,\infty)$. Hence comparing (\ref{e:convTeps})  and (\ref{e:convMu1}) we conclude that $M^{(u)}(1) \eqd M (u,\infty)$. Further, from (\ref{e:convas}) it follows that
 \begin{equation}\label{e:convh1}
M^{(u)}(1) \dto M \qquad \textrm{as} \ u \to 0,
\end{equation}
 which means that
\begin{equation}\label{e:convmdf}
 F_{u}(x) := \PP (M^{(u)}(1) \leq x ) \to F(x) := \PP (M \leq x) \qquad \textrm{as} \ u \to 0,
\end{equation}
for all $x \in \EE^{d}$ that are continuity points of $F$.
From (\ref{e:extrprdistr}) we obtain
$$ F^{t}_{u}(x) = \PP (M^{(u)}(t) \leq x)$$
for $t \in [0,1]$ and $x \in \EE^{d}$, which implies that the multivariate distribution function $F_{u}$ is max-infinitely divisible (cf. Resnick~\cite{Re07}, Section 5.6). Since the class of max-infinitely divisible distributions is closed in $\mathbb{R}^{d}$ with respect to weak convergence (cf. Proposition 5.1 in Resnick~\cite{Re87}), relation (\ref{e:convmdf}) implies that $F$ is max-infinitely divisible, and hence by Proposition 5.8 in Resnick~\cite{Re87} there exists an exponent measure $\mu$ on $\EE^{d}$ such that
$$F(x) = e^{-\nu ([[0,x]]^{c})}, \qquad x \in \EE^{d}.$$
Therefore, from (\ref{e:convmdf}) we obtain, as $u \to 0$,
$$\nu^{(u)}([[0,x]]^{c}) \to \nu([[0,x]]^{c})$$
for all continuity points $x$ of $\nu([[0,\,\cdot\,]]^{c})$. Now an application of Lemma 6.1 in Resnick~\cite{Re07} yields that $\nu^{(u)} \vto \nu$ as $u \to 0$. Therefore, by Lemma~\ref{l:fddmultconv} it follows that the finite dimensional distributions of $M^{(u)}(\,\cdot\,)$ converge to the finite dimensional distributions of $\widetilde{M}(\,\cdot\,)$ as $u \to 0$, where $\widetilde{M}(\,\cdot\,)$ is the extremal process generated by the Poisson process $T=\sum_{i}\delta_{(T_{i}, K_{i})}$ with mean measure $\lambda \times \nu$, i.e. $\widetilde{M}(t) = \bigvee_{T_{i} \leq t}K_{i}$, $t \in [0,1]$.


This implies that the finite dimensional distributions of each coordinate $M^{(u) k}(\,\cdot\,)$ ($k=1,\ldots,d)$ converge to the finite dimensional distributions of $\widetilde{M}^{k}(\,\cdot\,)$ as $u \to 0$.
According to the arguments used in the univariate case (see the proof of Theorem 4.3 in Krizmani\'{c}~\cite{Kr14}) this suffices to conclude that $M^{(u) k}(\,\cdot\,) \dto \widetilde{M}^{k}(\,\cdot\,)$ in $D([0,1], \mathbb{R}_{+})$ with the $M_{1}$ topology. Hence $\{M^{(u) k} : u>0 \}$ is tight, and thus by Lemma 3.2 in Whitt~\cite{Wh07} it follows that $\{M^{(u)} : u>0 \}$ is also tight (in the space $D([0,1], \mathbb{R}_{+}^{d})$ with the product topology generated by the metric $d_{p}$).

From the convergence of finite dimensional distributions and tightness for processes $M^{(u)}(\,\cdot\,)$ we obtain the convergence in distribution, i.e. as $u \to 0$,
\begin{equation}\label{e:Muconvd}
 M^{(u)}(\,\cdot\,) \dto \widetilde{M}(\,\cdot\,)
\end{equation}
in $D([0,1], \mathbb{R}_{+}^{d})$ with the $WM_{1}$ topology.

If we show that
$$ \lim_{u \to 0}\limsup_{n \to \infty} \PP(d_{p}(M_{n}(\,\cdot\,),M_{n}^{(u)}(\,\cdot\,)) > \epsilon)=0$$
for any $\epsilon >0$, from (\ref{eq:convaboveu}) and (\ref{e:Muconvd}) by a variant of Slutsky's theorem (see Theorem 3.5 in Resnick~\cite{Re07}) it will follow that
$ M_{n}(\,\cdot\,) \dto \widetilde{M}(\,\cdot\,)$ as $n \to \infty$,
in $D([0,1], \mathbb{R}_{+}^{d})$ with the $WM_{1}$ topology.

Since the
 metric $d_{p}$ on $D([0,1], \mathbb{R}_{+}^{d})$ is bounded above by the uniform metric on
 $D([0,1], \mathbb{R}_{+}^{d})$ (see Theorem 12.10.3 in Whitt~\cite{Whitt02}), it suffices to show that
 $$ \lim_{u \downarrow 0} \limsup_{n \to \infty} \PP \biggl(
 \sup_{0 \le t \le 1} \|M_{n}^{(u)}(t) - M_{n}(t)\| >
 \epsilon \biggr)=0.$$
 Recalling the definitions and using the inequality (\ref{e:maxineq}) , we have
 \begin{eqnarray*}
     \PP \bigg(
     \sup_{0 \le t \le 1} \|M_{n}^{(u)}(t) - M_{n}(t)\| >  \epsilon \bigg) & & \\[0.3em]
    & \hspace*{-20em} = & \hspace*{-10em} \PP \bigg(
       \sup_{0 \le t \le 1} \ \max_{k=1,\ldots,d} \bigg| \bigvee_{i=1}^{\lfloor nt \rfloor} \bigg( \frac{X_{i}^{k}}{a_{n}}
       1_{ \big\{ \frac{X_{i}^{k}}{a_{n}} > u \big\} }  - \bigvee_{i=1}^{\lfloor nt \rfloor}\frac{X_{i}^{k}}{a_{n}}
        \bigg) \bigg| > \epsilon
       \bigg)\\[0.4em]
     & \hspace*{-20em} \leq & \hspace*{-10em} \PP \bigg(
       \sup_{0 \le t \le 1} \ \max_{k=1,\ldots,d}  \bigvee_{i=1}^{\lfloor nt \rfloor} \frac{X_{i}^{k}}{a_{n}}
       1_{ \big\{ \frac{X_{i}^{k}}{a_{n}} \leq u \big\} }  > \epsilon
       \bigg)\\[0.4em]
      & \hspace*{-20em} = & \hspace*{-10em} \PP \bigg(
       \Big\| \bigvee_{i=1}^{n} \bigg( \frac{X_{i}^{k}}{a_{n}}
       1_{ \big\{ \frac{X_{i}^{k}}{a_{n}} \leq u \big\} }\bigg)_{k=1,\ldots,d} \Big\| > \epsilon
       \bigg)\\[0.4em]
      & \hspace*{-20em} \leq & \hspace*{-10em} \sum_{k=1}^{d} \PP \bigg(
        \bigvee_{i=1}^{n} \frac{X_{i}^{k}}{a_{n}}
       1_{ \big\{ \frac{X_{i}^{k}}{a_{n}} \leq u \big\} }  > \epsilon
       \bigg).
 \end{eqnarray*}
 Since the last term above is equal to zero for $ u \in (0, \epsilon)$, it holds that
 $$ \lim_{u \to 0}\limsup_{n \to \infty} \PP(d_{p}(M_{n}(\,\cdot\,),M_{n}^{(u)}(\,\cdot\,)) > \epsilon)=0,$$
 and this concludes the proof.
\end{proof}

\begin{rem}\label{r:j1m1}
The $WM_{1}$ convergence in Theorem~\ref{t:functconvergence} in
general can not be replaced by the standard $M_{1}$ convergence. This is
shown in Example~\ref{ex:WM1M1}.

The problem in our proof if we consider the standard $M_{1}$ topology is Lemma~\ref{l:contfunct}, which in this case does not hold.
To see this, fix $u>0$ and define
  $$ \eta_{n} = \delta_{(\frac{1}{2}- \frac{1}{n}, (2u,0))} + \delta_{(\frac{1}{2}-\frac{1}{2n}, (0,2u))} \qquad \textrm{for} \  n \geq 3.$$
 Then $\eta_{n} \vto \eta$, where
  $$\eta = \delta_{(\frac{1}{2}, (2u,0))} + \delta_{(\frac{1}{2}, (0,2u))} \in \Lambda.$$
 It is easy to compute
 $$ \phi^{(u)\,1}(\eta_{n})(t) = 2u\,1_{[\frac{1}{2} - \frac{1}{n}, 1]}(t) \quad \textrm{and} \quad \phi^{(u)\,2}(\eta_{n})(t) = 2u\,1_{[\frac{1}{2} - \frac{1}{2n}, 1]}(t).$$
 Then
 $$ y_{n} (t) := \phi^{(u)\,1}(\eta_{n})(t) - \phi^{(u)\,2}(\eta_{n})(t) = 2u\,1_{[\frac{1}{2} - \frac{1}{2n}, \frac{1}{2} - \frac{1}{n})}(t), \quad t \in [0,1],$$
 and similarly
 $$ y(t) := \phi^{(u)\,1}(\eta)(t) - \phi^{(u)\,2}(\eta)(t)=0, \quad t \in [0,1].$$

 For all parametric representations $(r_{n}, u_{n}) \in \Pi(y_{n})$ and $(r, u) \in \Pi(y)$ we have
 $$ \|u_{n} - u \|_{[0,1]} = 2u.$$
 Hence $d_{M_{1}}(y_{n}, y) \geq 2u$ for all $n \geq 3$, which means that $d_{M_{1}}(y_{n}, y)$ does not converge to zero as $n \to \infty$.
 Since
 $$ d_{M_{1}}(y_{n}, y) \leq d_{M_{1}}(\phi^{(u)}(\eta_{n}), \phi^{(u)}(\eta))$$
 (see Theorem 12.7.1 in Whitt~\cite{Whitt02}),
 we conclude that $d_{M_{1}}(\phi^{(u)}(\eta_{n}), \phi^{(u)}(\eta))$ does not converge to zero. Therefore the maximum functional $\phi^{(u)}$ is not continuous at $\eta$ with respect to the standard $M_{1}$ topology. Since $\eta \in \Lambda$ we conclude that $\phi^{(u)}$ is not continuous on the set $\Lambda$.

\end{rem}

\section{Examples}\label{s:four}

\begin{ex}\label{ex:WM1M1} (A $m$-dependent process).
Let $(Z_n)_{n \in \mathbb{Z}}$ be a sequence of i.i.d.~unit Fr\'{e}chet random variables, i.e. $\PP(Z_{n} \leq x) = e^{-1/x}$ for $x>0$. Hence $Z_{n}$ is regularly varying with index $\alpha=1$. Take a sequence of positive real numbers $(a_{n})$ such that
$n \PP (Z_{1}>a_{n}) \to 1$ as $n \to \infty$.
Now let
$$
X_n = (Z_n, Z_{n-1}, \ldots, Z_{n-m}), \quad {n \in \mathbb{Z}}.
$$
Then every $X_{n}$ is also regularly varying with index $\alpha=1$.
By an application of Proposition 5.1 in Basrak et al.~\cite{BDM02b} it can be seen that the random process $(X_{n})$ is jointly regularly varying. Since the sequence $(X_{n})$ is $m$--dependent, it follows immediately that Conditions~\ref{c:mixcond1} and~\ref{c:mixcond2} hold (cf. Basrak and Krizmani\'{c}~\cite{BaKr}).

Therefore $(X_{n})$ satisfies all the conditions of Theorem~\ref{t:functconvergence}, and the corresponding partial maxima process $M_{n}(\,\cdot\,)$ converge in distribution in $D([0,1], \mathbb{R}_{+}^{m+1})$ to an extremal process $\widetilde{M}(\,\cdot\,)$ under the weak $M_{1}$ topology.

Next we show that $M_{n}(\,\cdot\,)$ does not converge in distribution under the standard $M_{1}$ topology on $D([0,1], \mathbb{R}_{+}^{m+1})$. This shows that the weak $M_{1}$ topology in Theorem~\ref{t:functconvergence} in general can not be replaced by the standard $M_{1}$ topology. In showing this we use, with appropriate modifications, a combination of arguments used by Basrak and Krizmani\'{c}~\cite{BaKr} in their Example 4.1 and Avram and Taqqu~\cite{AvTa92} in their Theorem 1 (cf. also Example 5.1 in Krizmani\'{c}~\cite{Kr14}).

For simplicity take $m=1$.
We have $M_{n}(t) = (M_{n}^{1}(t), M_{n}^{2}(t))$, where
$$ M_{n}^{1}(t) = \bigvee_{j=1}^{\floor{nt}} \frac{Z_{j}}{a_{n}} \quad \textrm{and} \quad
 M_{n}^{2}(t) = \bigvee_{j=1}^{\floor{nt}} \frac{Z_{j-1}}{a_{n}}.$$
Let
 $$ V_{n}(t) := M_{n}^{1}(t) - M_{n}^{2}(t), \quad t \in [0,1].$$

The first step is to show that $V_{n}(\,\cdot\,)$ does not converge in distribution in $D([0,1], \mathbb{R}_{+})$ endowed with the (standard) $M_{1}$ topology. For this, according to Skorohod~\cite{Sk56} (cf. also Proposition 2 in Avram and Taqqu~\cite{AvTa92}), it suffices to show that
 \begin{equation}\label{e:osc1}
 \lim_{\delta \to 0} \limsup_{n \to \infty} \PP ( \omega_{\delta}(V_{n}(\,\cdot\,)) > \epsilon ) > 0
 \end{equation}
 for some $\epsilon >0$, where
 $$ \omega_{\delta}(x) = \sup_{{\footnotesize \begin{array}{c}
                                t_{1} \leq t \leq t_{2} \\
                                0 \leq t_{2}-t_{1} \leq \delta
                              \end{array}}
} M(x(t_{1}), x(t), x(t_{2}))$$
($x \in D([0,1], \mathbb{R}_{+}), \delta >0)$ and
$$ M(x_{1},x_{2},x_{3}) = \left\{ \begin{array}{ll}
                                   0, & \ \ \textrm{if} \ x_{2} \in [x_{1}, x_{3}], \\
                                   \min\{ |x_{2}-x_{1}|, |x_{3}-x_{2}| \}, & \ \ \textrm{otherwise},
                                 \end{array}\right.$$
Note that $M(x_{1},x_{2},x_{3})$ is the distance form $x_{2}$ to $[x_{1}, x_{3}]$, and $\omega_{\delta}(x)$ is the $M_{1}$ oscillation of $x$.

Let $i'=i'(n)$ be the index at which $\max_{1 \leq i \leq n-1}Z_{i}$ is obtained. Fix $\epsilon >0$ and introduce the events
 $$A_{n,\epsilon} = \{ Z_{i'} > \epsilon a_{n} \} = \Big\{ \max_{1 \leq i \leq n-1}Z_{i} > \epsilon a_{n}\Big\}$$
 and
 $$ B_{n,\epsilon} = \{Z_{i'}>\epsilon a_{n} \ \textrm{and} \ \exists\,l
 \neq 0, -i' \leq l \leq 1, \ \textrm{such that} \ Z_{i'+l} > \epsilon a_{n} / 4 \}.$$
Using the facts that $(Z_{i})$ is an i.i.d.~sequence and $n
 \PP(Z_{1}> c a_{n}) \to 1/c$ as $n \to \infty$ for $c>0$
 (which follows from the regular variation property of $Z_{1}$) we get
 \begin{equation}\label{e:limAn}
  \lim_{n \to \infty}\PP(A_{n,\epsilon}) = 1 - e^{-1/\epsilon} \end{equation}
and
 \begin{equation}\label{e:limBn}
 \limsup_{n \to \infty} \PP(B_{n,\epsilon})  \leq  \frac{4}{\epsilon^{2}}
 \end{equation}
(see Example 5.1 in Krizmani\'{c}~\cite{Kr14}).

On the event $A_{n,\epsilon} \setminus B_{n,\epsilon}$
one has $Z_{i'} > \epsilon a_{n}$ and $Z_{i'+l} \leq \epsilon a_{n}/4$ for every $l \neq 0$, $-i' \leq l \leq 1$, so that
$$ \bigvee_{j=1}^{i'}\frac{Z_{j}}{a_{n}} = \frac{Z_{i'}}{a_{n}} > \epsilon$$
and
$$ \max \bigg\{ \bigvee_{j=1}^{i'-1}\frac{Z_{j}}{a_{n}}, \bigvee_{j=1}^{i'}\frac{Z_{j-1}}{a_{n}}, \bigvee_{j=1}^{i'-1}\frac{Z_{j-1}}{a_{n}} \bigg\}   \leq \frac{\epsilon}{4}.$$
Therefore
$$ V_{n} \Big( \frac{i'}{n} \Big) = \bigvee_{j=1}^{i'}\frac{Z_{j}}{a_{n}} - \bigvee_{j=1}^{i'}\frac{Z_{j-1}}{a_{n}} > \epsilon - \frac{\epsilon}{4} = \frac{3 \epsilon}{4},$$
$$ V_{n} \Big( \frac{i'-1}{n} \Big) = \bigvee_{j=1}^{i'-1}\frac{Z_{j}}{a_{n}} - \bigvee_{j=1}^{i'-1}\frac{Z_{j-1}}{a_{n}} \in \Big[ - \frac{\epsilon}{4}, \frac{\epsilon}{4} \Big],$$
$$V_{n} \Big( \frac{i'+1}{n} \Big) = \bigvee_{j=1}^{i'+1}\frac{Z_{j}}{a_{n}} - \bigvee_{j=1}^{i'+1}\frac{Z_{j-1}}{a_{n}} = \frac{Z_{i'}}{a_{n}} - \frac{Z_{i'}}{a_{n}} = 0,$$
and these imply
\begin{equation}\label{e:inc1}
  \Big| V_{n} \Big( \frac{i'}{n} \Big) - V_{n} \Big( \frac{i'-1}{n} \Big) \Big|  > \frac{3 \epsilon}{4} - \frac{\epsilon}{4} = \frac{\epsilon}{2}
\end{equation}
and
\begin{equation}\label{e:inc2}
  \Big| V_{n} \Big( \frac{i'+1}{n} \Big) - V_{n} \Big( \frac{i'}{n} \Big) \Big| > \frac{3 \epsilon}{4}.
\end{equation}

Note that on the set $A_{n,\epsilon} \setminus B_{n,\epsilon}$ it also holds that
$$ V_{n} \Big( \frac{i'}{n} \Big) \notin \Big[ V_{n} \Big( \frac{i'-1}{n} \Big), V_{n} \Big( \frac{i'+1}{n} \Big) \Big],$$
which implies that
\begin{eqnarray*}
  M \Big( V_{n} \Big( \frac{i'-1}{n} \Big), V_{n} \Big( \frac{i'}{n} \Big), V_{n} \Big( \frac{i'+1}{n} \Big) \Big) & &  \\[0.8em]
   & \hspace*{-20em} =& \hspace*{-10em} \min \bigg\{  \Big| V_{n} \Big( \frac{i'}{n} \Big) - V_{n} \Big( \frac{i'-1}{n} \Big) \Big|, \Big| V_{n} \Big( \frac{i'+1}{n} \Big) - V_{n} \Big( \frac{i'}{n} \Big) \Big| \bigg\}.
\end{eqnarray*}
Taking into account (\ref{e:inc1}) and (\ref{e:inc2}) we obtain
\begin{eqnarray*}
  \omega_{2/n}(V_{n}(\,\cdot\,)) & = & \sup_{{\footnotesize \begin{array}{c}
                                t_{1} \leq t \leq t_{2} \\
                                0 \leq t_{2}-t_{1} \leq 2/n
                              \end{array}}
} M(V_{n}(t_{1}), V_{n}(t), V_{n}(t_{2})) \\[0.8em]
   & \geq & M \Big( V_{n} \Big( \frac{i'-1}{n} \Big), V_{n} \Big( \frac{i'}{n} \Big), V_{n} \Big( \frac{i'+1}{n} \Big) \Big) > \frac{\epsilon}{2}
\end{eqnarray*}
on the event $A_{n,\epsilon} \setminus B_{n,\epsilon}$. Therefore, since $\omega_{\delta}(\,\cdot\,)$ is nondecreasing in $\delta$, it holds that
 \begin{eqnarray}\label{e:oscM1}
  \nonumber \liminf_{n \to \infty} \PP(A_{n,\epsilon} \setminus B_{n,\epsilon}) & \leq & \liminf_{n \to \infty}
 \PP (\omega_{2/n} (V_{n}(\,\cdot\,)) >  \epsilon /2)\\[0.4em]
 & \leq &   \lim_{\delta \to 0} \limsup_{n \to \infty}  \PP (\omega_{\delta} (V_{n}(\,\cdot\,)) >  \epsilon/2).
 \end{eqnarray}

Note that $x^{2}(1-e^{-1/x})$ tends to infinity as $x \to \infty$, and therefore we can find $\epsilon >0$ such that  $\epsilon^{2}(1-e^{-1/ \epsilon}) > 4$, i.e.
 $$ 1-e^{-1/ \epsilon} > \frac{4}{ \epsilon^{2}}.$$
 For this $\epsilon$, by relations (\ref{e:limAn}) and (\ref{e:limBn}), it holds that
 $$\lim_{n \to \infty} \PP(A_{n,\epsilon}) > \limsup_{n \to \infty} \PP(B_{n,\epsilon}),$$
 i.e.
 $$  \liminf_{n \to \infty} \PP(A_{n,\epsilon} \setminus B_{n,\epsilon}) \geq \lim_{n \to \infty}\PP(A_{n,\epsilon}) - \limsup_{n \to \infty} \PP(B_{n,\epsilon}) >0.$$
Thus by (\ref{e:oscM1}) we obtain
$$ \lim_{\delta \to 0} \limsup_{n \to \infty}  \PP (\omega_{\delta} (V_{n}(\,\cdot\,)) >  \epsilon/2) > 0$$
and (\ref{e:osc1}) holds, i.e. $V_{n}(\,\cdot\,)$ does not converge in distribution in $D([0,1], \mathbb{R}_{+})$ endowed with the (standard) $M_{1}$ topology.

If $M_{n}(\,\cdot\,)$ would converge in distribution to some $\widetilde{M}(\,\cdot\,)$ in the standard $M_{1}$ topology on $D([0,1], \mathbb{R}_{+}^{2})$, then using the fact that linear combinations of the coordinates are continuous in the same topology (cf. Theorem 12.7.1 and Theorem 12.7.2 in Whitt~\cite{Whitt02}) and the continuous mapping theorem, we would obtain that $V_{n}(\,\cdot\,) = M_{n}^{1}(\,\cdot\,) - M_{n}^{2}(\,\cdot\,)$ converges to $\widetilde{M}^{1}(\,\cdot\,) - \widetilde{M}^{2}(\,\cdot\,)$ in $D([0,1], \mathbb{R}_{+})$ endowed with the standard $M_{1}$ topology, which is impossible, as is shown above.

\end{ex}

\begin{ex} (Stochastic recurrence equation).
Suppose the $d$--dimensional random process $(X_{n})$ satisfies a stochastic recurrence equation
$$X_{n} = A_{n}X_{n-1}+B_{n}, \quad n \in \mathbb{Z},$$
for some i.i.d.~sequence $((A_{n}, B_{n}))$ of random $d \times d$ matrices $A_{n}$ and $d$--dimensional vectors $B_{n}$, all with nonnegative components. Then it can be shown that under relatively general conditions the process $(X_{n})$ satisfies all conditions of Theorem~\ref{t:functconvergence} (see Example 4.2 in Basrak and Krizmani\'{c}~\cite{BaKr}), and hence the corresponding partial maxima process $M_{n}(\,\cdot\,)$ converges in $D([0,1], \mathbb{R}_{+}^{d})$ with the weak $M_{1}$ topology.
\end{ex}

\begin{ex} (Multivariate squared GARCH process).
We consider the multivariate GARCH (p, q) model with constant conditional correlations, which is defined as follows; see Fern\'{a}ndez and Muriel~\cite{FeMu09}. Let $(\eta_{n})_{n \in \mathbb{Z}}$ be a sequence of i.i.d.~random vectors with mean vector $0$ and covariance matrix $R$ such that
$R(i,i)=1$ for all $i= 1, \ldots, d$. The stochastic process $(X_{n})_{n \in \mathbb{Z}}$ is a CCC-GARCH (p,q) process if it satisfies the following equations
\begin{eqnarray*}
  \delta(H_{n})  &=& C + \sum_{i=1}^{p}A_{i} \delta(X_{n-i}X_{n-i}^{T}) + \sum_{j=1}^{q}B_{j} \delta (H_{n-j}), \\[0.3em]
  D_{n} &=& \textrm{diag} (H_{n}(1,1)^{1/2}, H_{n}(2,2)^{1/2}, \ldots, H_{n}(d,d)^{1/2}), \\[0.2em]
  H_{n} &=& D_{n}RD_{n}, \\[0.2em]
  X_{n} &=& D_{n} \eta_{n},
\end{eqnarray*}
where for a square $d \times d$ matrix $M$, $\delta (M)$ denotes the vector whose entries are $\delta(M)(i) = M(i,i)$ for $i=1,\ldots,d$ (i.e. the main diagonal of $M$), and $\textrm{diag}(M)$ denotes the diagonal matrix with the same diagonal as $M$. The vector $C$ is assumed to be positive and the matrices $A_{i}, B_{j}$ are assumed to be nonnegative for
$i = 1, \ldots, p$ and $j = 1, \ldots, q$.

Assume now the matrices $A_{i}, B_{j}$ have no zero rows, $\eta_{1}$ has a strictly positive density on $\mathbb{R}^{d}$ and for any $\gamma \geq 1$ there exists $h>1$ such that $\gamma^{h} \leq \mathrm{E}[(\eta_{1}^{j})^{2h}] \leq \infty$ for all $j=1,\ldots,d$. Put
$$ Y_{n} = (\delta(H_{n+1})^{T}, \ldots, \delta(H_{n-q+2})^{T}, \delta(X_{n}X_{n}^{T})^{T}, \ldots, \delta(X_{n-p+2}X_{n-p+2}^{T})^{T})^{T}.$$
Then by Theorem 5 in~\cite{FeMu09} there exists $\alpha >0$ such that for every $x \in \mathbb{R}^{d(p+q-1)} \setminus \{0\}$, $\sum_{i=1}^{d(p+q-1)}x^{i}Y_{1}^{i}$ is regularly varying with index $\alpha$. If $\alpha$ is not an even integer and $\eta_{1}$ has symmetric marginal distributions, from Corollary 6 in~\cite{FeMu09} we know that the process $(X_{n})$ is jointly regularly varying with index $2\alpha$. Further $(X_{n})$ is $\beta$--mixing (see Remark 4 in~\cite{FeMu09}, cf.~also Boussama~\cite{Bo05}), and since $\beta$--mixing implies strong mixing (cf.~Bradley~\cite{Br05}), Condition~\ref{c:mixcond1} holds. As in the one-dimensional case in Basrak et al.~\cite{BDM02b} it can be proved that $(X_{n})$ satisfies Condition~\ref{c:mixcond2}.

The joint regular variation property and Conditions~\ref{c:mixcond1} and~\ref{c:mixcond2} transfer immediately to the squared CCC-GARCH (p,q) process
$$ X_{n}^{2} = ((X_{n}^{1})^{2}, \ldots, (X_{n}^{d})^{2})$$
(with the remark that this process is jointly regularly varying with index $\alpha$), and from Theorem~\ref{t:functconvergence}, we conclude that the corresponding partial maxima process of $(X_{n}^{2})$ converges in distribution in $D([0,1], \mathbb{R}_{+}^{d})$ to an extremal process under the weak $M_{1}$ topology.

\end{ex}

\section*{Acknowledgements}
This work has been supported in part by Croatian Science Foundation under the project 3526.


\begin{thebibliography}{99}

\pagestyle{fancyplain}
 \rhead[\fancyplain{}{\textsc{Bibliography}}] %
    {\fancyplain{}{\bfseries \thepage}}
 \lhead[\fancyplain{}{\bfseries \thepage}] %
    {\fancyplain{}{\textsc{Bibliography}}}
\cfoot{\bfseries \thepage}


\bibitem{Ad78}
Adler, R. J., Weak convergence results for extremal processes generated by dependent random variables, {\em Ann. Probab.\/} {\bf 6} (1978), 660--667.

\bibitem{AvTa92}
Avram, F. and Taqqu, M., Weak convergence of sums of moving
averages in the $\alpha$--stable domain of attraction, {\em Ann.
Probab.} {\bf 20} (1992), 483--503.

\bibitem{BaKr}
Basrak, B., Krizmani\'{c}, D., A multivariate functional limit theorem in weak $M_{1}$ topology, {\em J. Theoret. Probab.\/} {\bf 28} (2015), 119--136.

\bibitem{BDM02b}
Basrak, B., Davis, R. A. and Mikosch, T., Regular variation
of GARCH processes, {\em Stoch. Process. Appl.} {\bf 99} (2002),
95--115.

\bibitem{BKS}
Basrak, B., Krizmani\'{c}, D. and Segers, J., A functional limit
theorem for partial sums of dependent random variables with
infinite variance, {\em Ann. Probab.} {\bf 40} (2012), 2008--2033.

\bibitem{BaSe}
Basrak, B. and Segers, J., Regularly varying multivariate time
series, {\em Stochastic Process. Appl.} {\bf 119} (2009), 1055--1080.

\bibitem{BGST04}
Beirlant, J., Goegebeur, Y., Segers, J. and Teugels, J., \emph{Statistics of Extremes: Theory and Applications,} Wiley Series in Probability and Statistics. Wiley, Chichester, 2004.

\bibitem{Bo05}
Boussama, F., \emph{Ergodicit\'{e}, m\'{e}lange et estimation dans les model\'{e}s GARCH,} Ph.~D.~Thesis, Universit\'{e} de Paris 7, 1998.


\bibitem{Br05}
Bradley, R.C., Basic Properties of Strong Mixing Conditions.
A Survey and Some Open Questions. {\em Probab. Surv.}, Vol.2 (2005),
107--144.

\bibitem{DaMi98}
Davis, R. A. and Mikosch, T., The sample autocorrelations of heavy--tailed processes with applications to ARCH, {\em Ann. Statist.} {\bf 26} (1998), 2049--2080.

\bibitem{FeMu09}
Fern\'{a}ndez, B. and Muriel, N., Regular variation and related results for the multivariate GARCH(p.q) model with constant conditional correlations, {\em J. Multivariate Anal.} {\bf 100} (2009), 1538--1550.

\bibitem{FT28}
Fischer, R. A. and Tippett, L. H. C., Limiting forms of the frequency of the largest or smallest member of a sample. Proc. Cambridge. Philos. Soc. 24 (1928), 180--190.

\bibitem{Gn43}
Gnedenko, B., Sur la distribution limite du terme maximum d'une s\'{e}rie al\'{e}atoire, {\em Ann. Math.} {\bf 2} (1943), 423--453.

\bibitem{Ha70}
de Haan, L., \emph{On Regular Variation and Its Application to the Weak Convergence of Sample Extremes,} Volume 32 of Mathematical Centre Tracts. Mathematisch Centrum, Amsterdam, 1970.

\bibitem{HaFe06}
de Haan, L. and Ferreira, A., \emph{Extreme Value Theory. An Introduction,} Springer Series in Operations Research and Financial Engineering. Springer, New York, 2006.

\bibitem{Kr14}
Krizmani\'{c}, D., Weak convergence of partial maxima processes in the $M_{1}$ topology, {\em Extremes} {\bf 17} (2014), 447--465.

\bibitem{Kr16}
Krizmani\'{c}, D., Functional weak convergence of partial maxima processes, {\em Extremes} {\bf 19} (2016), 7--23.

\bibitem{La64}
Lamperti, J., On extreme order statistics, {\em Ann. Math. Statist.} {\bf 35} (1964), 1726--1737.

\bibitem{Le74}
Leadbetter, M. R., On extreme values in stationary sequences, {\em Z. Wahrsch. verw. Gebiete\/} {\bf 28} (1974), 289--303.

\bibitem{Le76}
Leadbetter, M. R., Weak convergence of high level exceedances by a stationary sequence, {\em Z. Wahrscheinlichkeitstheorie und Verw. Gebiete} {\bf 34} (1976), 11--15.

\bibitem{Re87}
Resnick, S. I., \emph{Point Processes, Regular Variation and Point Processes,} Springer-Verlag, New York, 1987.

\bibitem{Re07}
Resnick, S. I., \emph{Heavy-Tail Phenomena: Probabilistic nad
Statistical Modeling,} Springer Science+Business Media LLC, New
York, 2007.

\bibitem{Sk56}
Skorohod, A. V., Limit theorems for stochastic processes,
{\em Theor. Probab. Appl.\/} {\bf 1} (1956), 261--290.

\bibitem{Whitt02}
Whitt, W., \emph{Stochastic-Process Limits,} Springer-Verlag LLC, New
York, 2002.

\bibitem{Wh07}
Whitt, W., Proofs of the martingale FCLT, {\em Probab. Surv.\/} {\bf 4} (2007), 268--302.


\end{thebibliography}
\end{document}